\documentclass{amsart}

\usepackage{fullpage,enumerate}

\usepackage{amsmath,amsxtra,amsthm,amssymb,xr,enumerate}
\usepackage[all,cmtip]{xy}
\newtheorem{theorem}{Theorem}[section]
\newtheorem{lem}[theorem]{Lemma}
\newtheorem{prop}[theorem]{Proposition}
\newtheorem{corollary}[theorem]{Corollary}
\newtheorem{defn}[theorem]{Definition}

\newtheorem{remark}[theorem]{Remark}

\numberwithin{equation}{section}

\newcommand{\tE}{\tilde{\mathbb{E}}}

\newcommand{\calH}{\mathcal{H}}

\newcommand{\calO}{\mathcal{O}}
\newcommand{\calL}{\mathcal{L}}
\newcommand{\fM}{\mathfrak{M}}
\newcommand{\QTr}{Q_{T,r}}
\newcommand{\T}{\mathfrak{T}}
\newcommand{\DD}{\mathbb D}
\newcommand{\NN}{\mathbb N}
\newcommand{\BB}{\mathbb B}

\renewcommand{\AA}{\mathbb A}
\newcommand{\ZZ}{\mathbb{Z}}
\newcommand{\QQ}{\mathbb{Q}}
\newcommand{\Zp}{\ZZ_p}
\newcommand{\Cp}{\mathbb{C}_p}
\newcommand{\Qp}{\QQ_p}
\newcommand{\AQp}{\AA_{F}^+}

\newcommand{\vp}{\varphi}
\newcommand{\Brig}{\BB^+_{F,\rm rig}}
\newcommand{\Bcris}{\BB_{\rm cris}}
\newcommand{\BdR}{\BB_{\rm dR}}

\newcommand{\Dcris}{\DD_{\mathrm{cris}}}
\newcommand{\Acris}{\AA_{\mathrm{cris}}}

\newcommand{\ft}{\mathfrak{t}}
\newcommand{\fs}{\mathfrak{c}}
\newcommand{\fc}{\mathfrak{u}}

\newcommand{\tfc}{\tilde{\mathfrak{u}}}

\newcommand{\vs}{\vspace{6pt}}
\DeclareMathOperator{\Gal}{Gal}
\DeclareMathOperator{\ord}{ord}

\DeclareMathOperator{\Iw}{Iw}
\DeclareMathOperator{\Tr}{Tr}
\DeclareMathOperator{\Fil}{Fil}
\DeclareMathOperator{\cor}{cor}
\DeclareMathOperator{\rank}{rank}
\DeclareMathOperator{\pr}{pr}
\DeclareMathOperator{\Tw}{Tw}

\DeclareMathOperator{\GL}{GL}
\DeclareMathOperator{\Tam}{Tam}
\newcommand{\HIw}{H^1_{\Iw}}

\begin{document}

\title[Tamagawa numbers over cyclotomic extensions]{Bounds on the Tamagawa numbers of a crystalline representation over towers of cyclotomic extensions}

\begin{abstract}
In this paper, we study the Tamagawa numbers of a crystalline representation over a tower of cyclotomic extensions under certain technical conditions on the representation. In particular, we show that we may improve the asymptotic bounds given in the thesis of Arthur Laurent in certain cases.
\end{abstract}

\author{Antonio Lei}
\address{Antonio Lei\newline
D\'epartement de Math\'ematiques et de Statistique\\
Universit\'e Laval, Pavillion Alexandre-Vachon\\
1045 Avenue de la M\'edecine\\
Qu\'ebec, QC\\
Canada G1V 0A6}
\email{antonio.lei@mat.ulaval.ca}
\thanks{The author's research is supported by the Discovery Grants Program 05710 of the Natural Sciences and Engineering Research Council of Canada}

\maketitle

\section{Introduction}
In \cite{blochkato}, Bloch and Kato defined the local Tamagawa number of a general motive. Furthermore, they formulated a conjecture relating the Tamagawa number to the size of the Shafarevich-Tate group, generalizing works of Birch and Swinnerton-Dyer on elliptic curves.

Fix an odd prime $p$. Let $T$ be a crystalline $\Zp$-representation of $G_F$, where $F$ is a finite unramified extension of $\Qp$.  We write $V=T\otimes\Qp$. The local Tamagawa number of $T$ at a prime outside $p$ is relatively well understood (see for example \cite{fontaineperrinriou94}). In this article, we study the local Tamagawa number of $T$ at $p$ over a tower of cyclotomic extensions of $F$. More specifically, we give asymptotic bounds on $\Tam_{F_n,\omega_n}(T)$ for $n\ge1$, where $F_n=F(\mu_{p^n})$ and $\omega_n$ is a basis of $\det_{\Zp}\left(\calO_{F_n}\otimes_{\calO_F}\Dcris(T)/\Fil^0\Dcris(T)\right)$, with $\Dcris(T)$ being the Dieudonn\'e module of $T$. We show in \S\ref{S:tamagawa} that under some technical hypotheses on $T$, we have the inequality
\begin{equation}
\label{eq:bound}
\Tam_{F_n,\omega_n}(T)\le \frac{p^{(r(p^{n-1}+p-2)+s_1(p-1)+s_2[F_n:F])d'}}{\det_{\Zp}(1-\vp|\Dcris(V))_{(p)}},
\end{equation}
for all $n\ge1$. Here, $d'$ is the dimension of $\Dcris(V)/\Fil^0\Dcris(V)$, $r$, $s_1$ and $s_2$ are some constants that depend only on the integrality of the Frobenius action $\vp$ on $\Dcris(V)$ and $(\star)_{(p)}$ denotes $p^{\ord_p(\star)}$.

The structure of the paper is as follows. We shall first of all introduce some notation and review some standard results from $p$-adic Hodge Theory in \S\ref{S:notation}. We then reprove a result of Perrin-Riou in \S\ref{S:Eg}, which is one of the key ingredient of her construction of the big exponential map in \cite{perrinriou94}. This result will be needed to analyse the denominators of the Bloch-Kato exponential map. In \S\ref{S:wach}, we review some results of Wach modules from \cite{berger04} and use them to prove a relation between crystalline classes over cyclotomic extensions and power series. Furthermore, we shall analyse the denominators of these power series. In \S\ref{S:tamagawa}, we combine these results to prove \eqref{eq:bound}. We note that in Arthur Laurent's thesis \cite{laurent}, a different asymptotic bound is given. We shall discuss how the two bounds compare in \S\ref{S:compare}. In particular, if $f=\sum a_nq^n$ is a modular form of even weight $k$ and that $\ord_p(a_p)\ge k/2$, it is possible to work out the constants $r$, $s_1$ and $s_2$ explicitly and that \eqref{eq:bound} gives
\[
\Tam_{F_n,\omega_n}(T_f(k/2))\le p^{\left(k/2(p-1)(p^{n-1}+1)-p^{n-1}(p-2)\right)e},
\]
where $T_f$ is the two dimensional representation of Deligne associated to $f$ and $e$ is the degree of the smallest extension of $\Qp$ that contains all $a_n$.

Finally, there are two appendices to this paper. The first one contains some elementary results on $p$-adic valuations of elements in $F_n$ and their relations to cyclotomic polynomials. The second one is a discussion on a consequence of our results in \S\ref{S:Eg} on $Q$-systems, which are important objects in the study of $p$-adic Gross-Zagier formulae in \cite{kobayashi13,kobayashi14,ota14}.

\section{Notation}\label{S:notation}
\subsection{Iwasawa algebra and power series}
Let $p$ be an odd prime, with $v_p$ the normalized $p$-adic valuation on $\overline{\Qp}$ with $v_p(p)=1$. We fix $F$ a finite unramified extension of $\Qp$.
 It is equipped with the Frobenius automorphism $\sigma$. For $1\le n\le \infty$, we denote the extension $F(\mu_{p^n})$ by $F_n$.  We write $\Gamma$ for the Galois group $\Gal(F_\infty/F)$. We have the cyclotomic character $\chi$ on the absolute Galois group $G_F$ of $F$. Fix a topological generator $\gamma$ of $\Gal(F_\infty/F_1)$ and write $u=\chi(\gamma)\in1+p\Zp$.
 
Let $\Lambda_F=\calO_F[[\Gamma]]=\calO_F[\Delta][[\gamma-1]]$ be the Iwasawa algebra of $\Gamma$ over $\calO_F$. We define  $\calH_F$ to be the set of elements in $\Qp[\Delta][[\gamma-1]]$ that converge when we replace $\gamma-1$ by any elements in the open $p$-adic unit disk.
 
Let $m\in\ZZ$ and $I\subset\ZZ$ a finite subset. For an integer $n\ge0$, define $\omega_{n,m}(\gamma)=(u^{-m}\gamma)^{p^n}-1$, $\omega_{n,I}(\gamma)=\prod_{m\in I}\omega_{n,m}(\gamma)$, $\Phi_{n,m}(\gamma)=\Phi_{p^n}(u^{-m}\gamma)$ and $\Phi_{n,I}(\gamma)=\prod_{m\in I}\Phi_{n,m}(\gamma)$, where $\Phi_{p^n}$ is the $p^n$-th cyclotomic polynomial. We also introduce the following notation of $p$-adic logarithms
\[
\ell_{m}=\frac{\log(\gamma)}{\log(u)}-m=\frac{\log(u^{-m}\gamma)}{\log(u)}.
\]

 Given $f\in \calH_F$ and $g\in \Qp[[\gamma-1]]$. We say that $g$ is divisible by $f$ (written $g|f$) if each $\Delta$-isotypic component of $f$ is divisible by $g$ over $\Qp[[\gamma-1]]$.

We define $\Acris$, $\Bcris$ and $\BdR$ to be the usual Fontaine rings (definitions can be found in \cite{berger03,berger04} for example). Recall that the ring $\tE^+=\displaystyle\varprojlim_{x\mapsto x^p} \calO_{\Cp}$  is equipped with the operator $\vp$  defined by $x\mapsto x^p$ and that $\vp$  is invertible on $\tE^+$. We fix an element $\epsilon=(1,\zeta_p,\zeta_{p^2},\ldots)\in\tE^+$, where $\zeta_{p^n}$ is a primitive $p^n$-th root of unity. For an integer $n$, we write $\epsilon_n=[\vp^{-n}(\epsilon)]\in W(\tE^+)$. Here, $W(\tE^+)$ is the ring of Witt vectors with coefficients in $\tE^+$ and $[\bullet]$ denotes the Teichmuller lift. Note that, when $n\ge0$, we have 
\begin{equation}\label{eq:epsilon}
\epsilon_{-n}=\epsilon_0^{p^n}\in 1+p^nW(\tE^+)
\end{equation}
since the first $n$ coordinates of $\epsilon^{p^n}$ are all $1$.

Let $\pi=[\epsilon_0]-1$. Define $\AQp=\calO_F[[\pi]]$, which we shall identify with the set of power series $\calO_F[[X]]$. Given $f=\sum a_m\pi^m$, we write $f^{\sigma^n}$ to be $\sum\sigma^n(a_m)\pi^m$ for $n\in\ZZ$. The action of $\vp:\pi\mapsto (1+\pi)^p-1$ is extended $\sigma$-linearly to $\AQp$.  In other words, if $F=\sum a_m\pi^m\in\AQp$, then $\vp(F)=\sum \sigma(a_m)((1+\pi)^p-1)^m$. Further, there is a standard left inverse $\psi$ of $\vp$ on $\AQp$, which sends $(1+\pi)^i$ to $(1+\pi)^{i/p}$ if $p|i$ and $0$ otherwise.

 We define $\Brig$ to be the set of power series in $F[[\pi]]$ that converge on the open unit disk. We have $t=\log(1+\pi)\in\Brig$. The operators $\vp$ and $\psi$ on $\AQp$ extend naturally to $\Brig$. For an integer $m\ge0$, we  have the operator on $(\Brig)^{\psi=0}$ defined by $\nabla_m=t\frac{d}{dt}-m=t\partial-m$, where $\partial=(1+\pi)\frac{d}{d\pi}$. We note that the operator $\partial$ acts bijectively on $(\AQp)^{\psi=0}$ as well as $(\Brig)^{\psi=0}$.

We have an isomorphism of $\Lambda_F$-modules (the Mellin transform)
\begin{align*}
\fM_F:\calH_F&\stackrel{\sim}{\longrightarrow}(\Brig)^{\psi=0}\\
1&\mapsto1+\pi,
\end{align*}
where the image of $\Lambda_F$ is $(\AQp)^{\psi=0}$. If $f\in\calH_F$, then
\begin{equation}\label{eq:twisting}
\fM_F\left(\ell_m\cdot f\right)=\nabla_m\circ\fM_F(f).
\end{equation}

\begin{theorem}
Let $f\in\calH_F$ and $n\ge1$ an integer. We have $\Phi_{n-1}(\gamma)|f$ over $\calH_F$ if and only if $\fM_F(f)\in\Phi_n(1+\pi)(\Brig)^{\psi=0}$.
\end{theorem}
\begin{proof}
This is \cite[Theorem~5.4]{leiloefflerzerbes10}.
\end{proof}

\begin{corollary}\label{cor:mellin}
Let $f\in\calH_F$ and $m,n\ge1$ integers. We have $\Phi_{n-1,m}(\gamma)|f$ over $ 
\calH_F$ if and only if $\partial^m\fM_F(f)\in\Phi_n(1+\pi)(\Brig)^{\psi=0}$.
\end{corollary}
\begin{proof}
This follows from \cite[proof of Lemma~5.9]{leiloefflerzerbes10}.
\end{proof}

\subsection{Crystalline representations and  Wach modules}\label{S:crystalline}
Let $T$ be a crystalline rank-$d$ $\Zp$-representation of $G_F$ and $V=T\otimes_{\Zp}\Qp$. We assume that $T$ has Hodge-Tate weights in $[a;b]$ with  $b\ge1$. Furthermore, we shall assume the following Fontaine-Laffaille condition from \cite{fontainelaffaille} holds throughout.

\vs

\noindent
\underline{\textbf{(H.FL)}} $b-a\le p-1$.

\vs

We define $\HIw(F,T)$ to be the inverse limit $
\varprojlim H^1(F_n,T)
$, where the connecting map is the corestriction map. For $m\ge n\ge0$, we write $\cor_{m/n}$ for the corestriction map from $H^1(F_m,T)$ to $H^1(F_n,T)$. Let
\[
\pr_{T,n}:\HIw(F,T)\longrightarrow H^1(F_n,T)
\]
denote the natural projection.

Let $m\in\ZZ$. We have the Tate twist $T(m)=T\otimes \Zp e_m$, where $G_F$ acts on  $\Zp e_m$ via the character $\chi^m$. We have a natural isomorphism
\[
\HIw(F,T)\stackrel{\cdot e_m}{\longrightarrow}\HIw(F,T(m)).
\]

We write $\Dcris(T)$ for the Dieudonn\'e module of $T$. We assume the following hypothesis holds.

\vs

\noindent
\underline{\textbf{(H.eigen)}} None of the eigenvalues of $\vp$ on $\Dcris(T)$ is an integral power of $p$.

\vs

Note in particular that (H.eigen) implies that $T(m)^{G_{F_\infty}}=0$ for all $m\in\ZZ$. We have the identification
\begin{equation}\label{eq:twistDieudonne}
\Dcris(T(m))=\Dcris(T)\cdot t^{-m}e_m.
\end{equation}
We write $\Tw_m$ for the bijection
\begin{align*}
\Tw_m:(\Brig)^{\psi=0}\otimes\Dcris(T)&\rightarrow(\Brig)^{\psi=0}\otimes\Dcris(T(m))\\
f\otimes v&\mapsto \partial^{-m}f\otimes v\cdot t^{-m}e_m.
\end{align*}
Here, the tensor product is taken over $\calO_F$. We shall abuse notation and write $\vp$ for the diagonal map $\vp\otimes\vp$ on $\Brig\otimes\Dcris(T)$.

\begin{remark}\label{rk:commute}It can be checked that $\Tw_m$ commutes with the actions of  $\vp=\vp\otimes\vp$ on $(\Brig)^{\psi=0}\otimes\Dcris(T)$ and  $(\Brig)^{\psi=0}\otimes\Dcris(T(m))$ since $\partial^m\vp=p^m\vp\partial^m$ and $\vp(t^{-m})=p^{-m}t^{-m}$.\end{remark}

The tangent space of $V$ over $F_n$ is defined to be
\[
t_n(V):=F_n\otimes \Dcris(T)/\Fil^0\Dcris(T).
\]
When $n=0$, we simply write $t(V)=F\otimes \Dcris(T)/\Fil^0\Dcris(T)$.
Recall that the Bloch-Kato exponential map 
\[
\exp_{V,n}:t_n(V)\oplus\Dcris(T){\longrightarrow} H^1_f(F_n,V)
\]
is defined as the connecting map of the short exact sequence
\begin{equation}
0\rightarrow \Qp\rightarrow \Bcris\rightarrow\Bcris\oplus\BdR/\BdR^+\rightarrow0
\label{eq:exact}
\end{equation}
and $H^1_f(F_n,V)$ can be realized as
\[
\ker\left(H^1(F_n,V)\rightarrow H^1(F_n,V\otimes\Bcris)\right).
\]
We write $\exp^*_{V,n}:H^1(F_n,V)/H^1_f(F_n,V)\longrightarrow F_n\otimes\Fil^0\Dcris(V)$ for the dual map of $\exp_{V,n}$.
Note that under (H.eigen), we in fact have
\[
H^1_e(K,T)=H^1_f(K,T)=H^1_g(K,T)
\]
for all finite extension $K$ of $F$ (see \cite[\S3.11]{blochkato}).

We write  $\NN(T)$ for the Wach module of $T$ (see \cite[\S A]{berger03} or \cite[\S II]{berger04} for a definition). We recall that there is an identification
\begin{equation}\label{eq:twistWach} 
\NN(T(m))=\NN(T)\otimes \pi^{-m}e_m
\end{equation}
and that there is an isomorphism (\cite[Theorem A.3]{berger03})
\begin{equation} \label{eq:Herr}
h^1_T:\DD(T)^{\psi=1}=(\pi^a\NN(T))^{\psi=1}\stackrel{\sim}{\longrightarrow}\HIw(F,T).
\end{equation}
We may extend this map to
\[
h^1_T:\DD_{\rm rig}(V)^{\psi=1}\stackrel{\sim}{\longrightarrow}\calH_F\otimes\HIw(F,T),
\]
where $\DD_{\rm rig}(V)$ is the $(\vp,\Gamma)$-module over the Robba ring associated to $V$.
For an integer $n\ge0$, we shall write $h^1_{T,n}$ for the map $h^1_T$ composed with the natural projection $\pr_{T,n}$.

If $x\in F_n((t))\otimes\Dcris(T)$, we define $\partial_T(x)\in F_n\otimes\Dcris(T)$ to be the constant term (coefficient of $t^0$) of $x$. Recall that we have the map $\vp^{-n}:\Brig\rightarrow F_n[[t]]$ sending $\pi$ to $\zeta_{p^n}\exp(t/p^n)-1$.

\begin{theorem}\label{thm:berger}
If $y\in\DD_{\rm rig}(V)^{\psi=1}\cap(\Brig[t^{-1}]\otimes\Dcris(T))$, then
\[
\exp^*_{V,n}\circ h^1_{T,n}(y)=p^{-n}\partial_T(\vp^{-n}(y))
\]
for all $n\ge1$. If $n=0$, then 
\[
\exp^*_{V,0}\circ h^1_{T,0}(y)=(1-p^{-1}\vp^{-1})\partial_T(y).
\]
\end{theorem}
\begin{proof}
This is \cite[Theorem II.6]{berger03}
\end{proof}

\subsection{Perrin-Riou's exponential map}\label{S:PR}
Let $T$ be as above with both (H.FL) and (H.eigen) hold. Since $\Fil^{-b}\Dcris(T)=\Dcris(T)$, we have the Perrin-Riou exponential map  from \cite{perrinriou94}
\[
\Omega_{T,b}:(\Brig)^{\psi=0}\otimes\Dcris(T)\longrightarrow \calH_F\otimes\HIw(F,T).
\]
Note that $\Omega_{T,b}$ is a $\Lambda_F$-homomorphism. It is compatible under twisting, that is, we have the following commutative diagram:
\begin{equation}\label{eq:PRcommute}
\xymatrix{
      (\Brig)^{\psi=0}\otimes\Dcris(T)\ar[rrrr]^{\Omega_{T,b}}\ar[d]^{\Tw_m} &&&&\calH_F\otimes\HIw(F,T)\ar[d]^{\cdot e_m}\\
      (\Brig)^{\psi=0}\otimes\Dcris(T(m))\ar[rrrr]^{(-1)^{m}\Omega_{T(m),b+m}} &&&& \calH_F\otimes\HIw(F,T)
    }
\end{equation}
for any integer $m\ge 1-b$ (c.f. \cite[Th\'eor\`eme~3.2.3 (B.ii)]{perrinriou94}). If $g\in(\Brig)^{\psi=0}\otimes\Dcris(T)$, then by \cite[Proposition~2.2.1]{perrinriou94}, there exists a unique $G\in(\Brig\otimes\Dcris(T))^{\psi=1}$ such that $(1-\vp)G=g$ (see also \S\ref{S:Eg} below for a discussion on the solution to this equation).
By \cite[Theorem II.13]{berger03}, we can describe the Perrin-Riou exponential map via
\begin{equation}\label{eq:bergerexpo}
\Omega_{T,b}(g)=h^1_T\circ\nabla_{b-1}\circ\cdots\circ\nabla_0 (G).
\end{equation}


\begin{defn}\label{defn:Sigman}
We define the maps
\begin{align*}
\Xi_{T,n}:(\AQp)^{\psi=0}\otimes\Dcris(T)\rightarrow& H^1_f(F_n,V)\\
g\mapsto&\exp_{V,n}\left( (p\otimes\vp)^{-n}G^{\sigma^{-n}}(\zeta_{p^n}-1)\right),\\
\Sigma_{T,n}:(\AQp)^{\psi=0}\otimes\Dcris(T)\rightarrow& H^1_f(F_n,V)\\
g\mapsto& p^{(b-1)n}\exp_{V,n}\left( G^{\sigma^{-n}}(\zeta_{p^n}-1)\right),\\
\Theta_{T,n}:(\AQp)^{\psi=0}\otimes\Dcris(T)\rightarrow& t_n(V),\\
g\mapsto& G^{\sigma^{-n}}(\zeta_{p^n}-1)\mod\Fil^0\Dcris(T),
\end{align*}
where $G\in(\Brig\otimes\Dcris(T))^{\psi=1}$ is such that $(1-\vp)g=G$ and $G^{\sigma^{-n}}$ denotes $(\sigma^{-n}\otimes 1)G$.
\end{defn}

 For $j\le b-1$ and $n\ge1$, we have the interpolation formula
\begin{equation}\label{eq:PRXi}
\pr_{T(-j),n}(\Omega_{T,b}(g)\cdot e_{-j})=(-1)^j(b-j-1)!\Xi_{T(-j),n}\circ\Tw_{-j}(g)
\end{equation}
by \cite[Th\'eor\`eme 3.2.3 (A)]{perrinriou94} (after changing the sign of $j$).

\begin{prop}\label{prop:IntegralPR}
There exists an integer $s$ such that the image of  $\Sigma_{T,n}$ lies inside $p^{-s}H^1_f(F_n,T)$ for all $n\ge1$. Furthermore, we have the relation \[\cor_{n+1/n}\Sigma_{T,n+1}(g)=p^b\Sigma_{T,n}(1\otimes \vp (g))\]
for all $g\in (\AQp)^{\psi=0}\otimes\Dcris(T)$ and $n\ge1$.
\end{prop}
\begin{proof}
This is \cite[Proposition~2.4.2]{perrinriou94} after a slight modification of the definition of $\Sigma_{T,n}$ (from page 108 of \textit{op. cit.}).
\end{proof}

Assume that $a=0$, then  we have the $p$-adic regulator map
\begin{align}
\calL_{T}:\HIw(F,T)&\rightarrow\calH_{F}\otimes\Dcris(T)\notag\\
z&\mapsto(\fM_F^{-1}\otimes 1)\circ(1-\vp)\circ (h^1_T)^{-1}(z).\label{eq:PRregulator}
\end{align}
See \cite[\S3]{leiloefflerzerbes11} for a detailed discussion on this map. By \eqref{eq:twisting} and \eqref{eq:bergerexpo}, we have the relation
\begin{equation}\label{eq:regulator}
\calL_{T}=\ell_{b-1}\cdots \ell_0\cdot(\fM_F^{-1}\otimes1)\circ\Omega_{T,b}^{-1}.
\end{equation}



 \section{Analysis of the equation $(1-\vp)G=g$}\label{S:Eg}

In this section, we take $T$ to be a fixed crystalline representation as in \S\ref{S:crystalline} satisfying (H.FL) and (H.eigen). The Hodge-Tate weights of $T$ are in $[a;b]$ with  $b\ge1$. By (H.FL), $p^b\vp\Dcris(T)\subset \Dcris(T)$. We fix an integer $1\le r\le b$  such that the slopes of the action of $\vp$ on $\Dcris(T)$ are $\ge -r$. In particular, there exists an integer $s\ge0$ such that 
\begin{equation}
(p^r\vp)^k(\Dcris(T))\subset p^{-s}\Dcris(T)
\label{eq:eigen}
\end{equation}
for all $k\ge0$. We fix an $\calO_F$-basis $v_1,\ldots,v_d$ of $\Dcris(T)$.

Note that $1-p^j\vp$ is invertible on $F\otimes\Dcris(T)$ for all integers $j$ thanks to our assumption (H.eigen). We choose $s$ such that in addition to $\eqref{eq:eigen}$, we have
\begin{equation}\label{eq:newlattice}
(1-p^j\vp)^{-1}(\Dcris(T))\subset p^{-s}\Dcris(T)
\end{equation}
for all $0\le j\le r$.

Let $g\in(\Brig)^{\psi=0}\otimes\Dcris(T)$, our goal in this section is to  study the equation
\begin{equation}
(1-\vp)G=g.\tag{$E_g$}\label{eq:Eg}
\end{equation}
for $G\in(\Brig\otimes\Dcris(T))^{\psi=1}$.
The following is a slight modification of \cite[Proposition~2.2.1]{perrinriou94}.
\begin{prop}\label{prop:PR}
Let $T$, $r$ and $s$ be as above (so that \eqref{eq:eigen} and \eqref{eq:newlattice} are satisfied). Then,
\begin{enumerate}[(i)]
\item The map $(1-\vp):\Brig\otimes\Dcris(T)\rightarrow \Brig\otimes\Dcris(T)$ is bijective;
\item For all $g\in\AQp\otimes\Dcris(T)$ and $n\ge1$, $\tau\in G_{F_n}$, we have
\[
p^{(r-1)n}\cdot(\tau-1)G(\epsilon_n-1)\in p^{1-r-s}\Acris\otimes \Dcris(T),
\]
where $G$ is the unique solution to \eqref{eq:Eg} (which exists by part (i)).
\end{enumerate}
\end{prop}
\begin{proof}
(i) follows from  \cite[Proposition~2.2.1, parts (i) and (ii)]{perrinriou94} and (H.eigen). We rewrite the construction of $G$ in \textit{op. cit.}  on replacing $h$, which is the largest Hodge-Tate weight of $T$ in Perrin-Riou's proof, by our chosen integer $r$.

Firstly, assume that $g\in \pi^{r+1}\Brig\otimes \Dcris(T)$. Write $g=\pi^{r+1}\sum_{i=1}^df_i\otimes v_i$ where $f_i\in\Brig$. For an integer $k\ge0$, 
\[
\vp^k(g)=p^{-rk}\vp^{k}(\pi)^{r+1}\sum_{i=1}^d\vp^k(f_i)\otimes (p^r\vp)^k(v_i).
\]
Note that $(p^r\vp)^k(v_i)\in p^{-s} \Dcris(T)$ for any $k\ge0$ by \eqref{eq:eigen} and that $p^{-rk}\vp^k(\pi)^{r+1}\rightarrow 0$ as $k\rightarrow\infty$. Furthermore, $\vp^k(f_i)\rightarrow f_i(0)$ as $k\rightarrow\infty$. Hence, we deduce that $\vp^k(g)\rightarrow 0$. This implies that $\sum_{k=0}^\infty \vp^k(g)$ converges to a solution to \eqref{eq:Eg}.

Now, for a general $g\in(\Brig)^{\psi=0}\otimes\Dcris(T)$ and $j\ge0$ an integer, we define $\Delta_j(g)=(\partial^j\otimes 1)(g)|_{\pi=0}\in F\otimes\Dcris(T)$.  Then,
\[
\tilde{g}:=g-\sum_{j=0}^r\frac{t^j}{j!}\otimes\Delta_j(g)\in\pi^{r+1}\Brig\otimes\Dcris(T).
\]
By the first part of our proof, $\sum_{k=0}^{\infty}\vp^k (\tilde{g})$ converges and the limit is the solution to ($E_{\tilde{g}}$). Since $\vp(t^j)=p^jt^j$, we see that
\[
(1-\vp)\left(t^j\otimes(1-p^j\vp)^{-1}\Delta_j(g)\right)=t^j\otimes\Delta_j(g).
\]
Therefore, the solution to \eqref{eq:Eg} is given by
\begin{equation}\label{eq:solution}
G=\sum_{k=0}^{\infty}\vp^k( \tilde{g})+\sum_{j=0}^r\frac{t^j}{j!}\otimes(1-p^j\vp)^{-1}\Delta_j(g).
\end{equation}

We now prove (ii). We shall show that on replacing $\pi$ by $\epsilon_n-1$, the action of $p^{(r-1)n}(\tau-1)$ sends each individual term of \eqref{eq:solution} to $p^{1-r-s}\Acris\otimes\Dcris(T)$ for all $\tau\in G_{F_n}$. 

Fix some $\tau\in G_{F_n}$. By definition, $\chi(\tau)\in 1+p^n\Zp$. Let $\alpha=(\chi(\tau)-1)p^{-n}\in\Zp$. We may describe the actions of $\tau$ explicitly: $\tau\cdot t=\chi(\tau)t$ and $\tau\cdot\epsilon_k=\epsilon_k^{\chi(\tau)}=\epsilon_k\cdot\epsilon_{k-n}^\alpha$ for all integers $k$.

Let us first analyze the terms in the second sum in \eqref{eq:solution}. Since $\log(\epsilon_n)=p^{-n}t$, we have for $0\le j\le r$ the equation
\begin{align*}
&p^{(r-1)n}(\tau-1)\left(\frac{\log^j(\epsilon_n)}{j!}\otimes(1-p^j\vp)^{-1}\Delta_j(g)\right)\\
=&\frac{p^{(r-1)n}(\chi(\tau)^j-1)(p^{-n}t)^j}{j!}\otimes (1-p^j\vp)^{-1}\Delta_j(g)\\
=&\frac{p^{(r-j-1)n}(\chi(\tau)^j-1)t^j}{j!}\otimes (1-p^j\vp)^{-1}\Delta_j(g).
\end{align*}
Note that $\frac{p^{(r-j-1)n}(\chi(\tau)^j-1)t^j}{j!}\in \Acris$ since $\chi(\tau)^j-1\in p^n\Zp$ and $j\le r\le b\le p-1$. Therefore, thanks to \eqref{eq:newlattice}, we deduce that the above expression lies inside $p^{-s}\Acris\otimes\Dcris(T)\subset p^{1-r-s}\Acris\otimes\Dcris(T)$.

We now consider the terms that appear in the first sum in \eqref{eq:solution}:
\[
\vp^k(\tilde{g})=\vp^k(g)-\sum_{j=0}^r\left(\frac{p^{jk}t^j}{j!}\otimes \vp^k (\Delta_j(g))\right),
\]
where $k\ge0$.

\noindent
\underline{\textbf{(Case 1) $n\ge k$}:}

Let $g=\sum_{i=1}^d f_i\otimes v_i$, where $f_i\in\AQp$. Then
\begin{align*}
p^{(r-1)n}(\tau-1)\vp^k(g)(\epsilon_n-1)=&p^{(r-1)n}\sum_{i=1}^d\left( f_i^{\sigma^k}(\epsilon_{n-k}^{\chi(\tau)}-1)-f_i^{\sigma^k}(\epsilon_{n-k}-1)\right)\otimes \vp^k(v_i)\\
=& p^{(r-1)n}\sum_{i=1}^d\left( f_i^{\sigma^k}(\epsilon_{n-k}\cdot\epsilon_{-k}^{\alpha}-1)-f_i^{\sigma^k}(\epsilon_{n-k}-1)\right)\otimes\vp^k(v_i).
\end{align*}
From \eqref{eq:epsilon}, $\epsilon_{-k}\equiv1\mod p^k\Acris$, so the expression above lies inside
$
p^{(r-1)n+k}\Acris\otimes\vp^k(\Dcris(T))$. But $n\ge k$, so via \eqref{eq:eigen}, we deduce that this set is contained in  $\Acris\otimes p^{rk}\vp^k(\Dcris(T))\subset p^{-s}\Acris\otimes \Dcris(T)$.

If $0\le j\le r$, then
\begin{align*}
&p^{(r-1)n}(\tau-1)\left(\frac{p^{jk}\log^j(\epsilon_n)}{j!}\otimes\vp^k(\Delta_j(g))\right)\\
=&p^{(r-1)n+jk}\frac{(\chi(\tau)^j-1)p^{-jn}t^j}{j!}\otimes \vp^k(\Delta_j(g)).
\end{align*}
But $(\chi(\tau)^j-1)\in p^n\Zp$, so  the above expression is contained in
$
 p^{rn+jk-jn}\Acris\otimes\vp^k(\Delta_j(g))$,
 which is contained in $ \Acris\otimes(p^r\vp)^k(\Delta_j(g))\subset p^{-s}\Acris\otimes\Dcris(T)$ thanks to 
the fact that $n\ge k$ and \eqref{eq:eigen}. 

\noindent
\underline{\textbf{(Case 2) $n<k$}:}

Let $\tilde{g}=\sum_{i=1}^d f_i\otimes v_i$. Fix $0\le i\le d$ and write $f_i=\sum_{m=r+1}^\infty a_m\pi^m$, where $a_m\in\calO_F$. By the definition of $\tilde{g}$, we have $v_p(a_m)\ge -rs_m$ for all $m\ge r+1$, where $s_m\ge0$ is the integer such that $p^{s_m}$ is the largest $p$-power $\le m$.

Consider
\[
(\tau-1)\vp^k(f_i)(\epsilon_n-1)=f_i^{\sigma^k}(\epsilon_{n-k}^{\chi(\tau)}-1)-f_i^{\sigma^k}(\epsilon_{n-k}-1).
\]
Let $x=\epsilon_{n-k}-1\in p^{k-n}\Acris$ and $y=\epsilon_{n-k}^{\tau(\chi)}-\epsilon_{n-k}=\epsilon_{n-k}(\epsilon_{-k}^{\alpha}-1)\in p^k\Acris$ (c.f. \eqref{eq:epsilon}). We  have
\[
(\tau-1)\vp^k(f_i)(\epsilon_n-1)=\sum_{m=0}^\infty {\sigma^k}(a_m)((x+y)^m-x^m).
\]
For each $m\ge r+1$, $(x+y)^m-x^m\in p^{m(k-n)+n}\Acris$. Recall that $a_m\in p^{-rs_m}\calO_F$, so 
\[
{\sigma^k}(a_m)((x+y)^m-x^m)\in p^{-rs_m+m(k-n)+n}\Acris.
\]
Therefore,  we deduce that
\[
p^{(r-1)n}(\tau-1)\vp^{k}(\tilde{g})\in p^{rn-rs_m+m(k-n)-rk}\Acris\otimes(p^r\vp)^{k}\left(\Dcris(T)\right).
\]
But $m>r$, $k>n$ and $m\ge p^{s_m}$, so
\[
rn-rs_m+m(k-n)-rk=(m-r)(k-n)-rs_m\ge m-r-rs_m\ge m-r\left(\frac{\log(m)}{\log(p)}+1\right).
\]
The function $f(m)=m-r\left(\frac{\log(m)}{\log(p)}+1\right)$ is increasing by studying its derivative and the fact that $m>r$. Therefore, $f(m)\ge f(r+1)=1-\frac{r\log(r+1)}{\log(p)}$. Recall that $r\le b\le p-1$ by (H.FL), so $f(m)\ge 1-r$. Hence \[
p^{(r-1)n}(\tau-1)\vp^{k}(\tilde{g})\in p^{1-r}\Acris\otimes(p^r\vp)^{k}\left(\Dcris(T)\right)\subset p^{1-r-s}\Acris\otimes\Dcris(T)
\]
by \eqref{eq:eigen}.
\end{proof}

\begin{corollary}\label{cor:integralexp}
If $g\in(\AQp)^{\psi=0}\otimes\Dcris(T)$, then 
\[
p^{(r-b)n}\Sigma_{T,n}(g)\in p^{1-r-s}H^1_f(F_n,T).
\]
\end{corollary}
\begin{proof}
Recall from Definition~\ref{defn:Sigman} that 
\[
p^{(r-b)n}\Sigma_{T,n}(g)=p^{(r-1)n}\exp_{V,n}\left(G^{\sigma^{-n}}(\zeta_{p^n}-1)\right).
\]
But $\exp_{V,n}$ is defined to be the connecting map that arises from the fundamental short exact sequence
\[
0\rightarrow \Qp\rightarrow \Bcris^{\vp=1}\rightarrow \BB_{\rm dR}/\BB_{\rm dR}^+\rightarrow0
\]
after tensoring by $V$, which implies that $\exp_{V,n}\left(G^{\sigma^{-n}}(\zeta_{p^n}-1)\right)$ is represented by the cocycle that sends $\tau\in G_{F_n}$ to $(\tau-1)\cdot G^{\sigma^{-n}}(\epsilon_{n}-1)$.
Therefore, the lemma follows from Proposition~\ref{prop:PR}(ii).
\end{proof}


\section{On an inverse of $\Sigma_{T,n}$}\label{S:wach}

\subsection{Preliminary Results on Wach modules}

In this section, we fix a representation as in \S\ref{S:crystalline} under the conditions (H.eigen) and (H.FL). Let $\T$ be the Tate twist $T(-a)$ of $T$. So, all the Hodge-Tate weights of $\T$ are  $\ge0$.  Let $r_1=0,\ldots,r_d=b-a$ be the Hodge-Tate weights of $\T$. Let $v_1,\ldots,v_d$ be a $\calO_F$-basis of $\Dcris(\T)$ that respects its filtration. Then, by (H.FL), the matrix of $\vp$ with respect to this basis is of the form
\[
A:=\begin{pmatrix}
p^{-r_1}&&\\
&\ddots&\\
&&p^{-r_d}
\end{pmatrix}\cdot A_0,
\]
where $A_0\in\GL_d(\calO_F)$.

\begin{prop}\label{prop:Wachbasis}
Under the identification $\NN(\T)/\pi\NN(\T)=\Dcris(\T)$, we may lift $v_1,\ldots, v_d$ to  a $\AQp$-basis of $\NN(\T)$, with respect to which the matrix of $\vp$ is of the form
\[
P:=\begin{pmatrix}
\mu^{r_d-r_1} q^{-r_1}&&\\
&\ddots&\\
&& q^{-r_d}
\end{pmatrix}\cdot A_0,
\]
where $\mu=p/(q-\pi^{p-1})\in (\AQp)^\times$.
\end{prop}
\begin{proof}
Consider the representation $\T(-r_d)$. Its Hodge-Tate weights are $r_1-r_d,\ldots,0$ and $v_i\otimes t^{r_d}e_{-r_d}$ form a basis of $\Dcris(\T(-r_d))$. The matrix of $\vp$ with respect to this basis is
\[
p^{r_d}A=\begin{pmatrix}
p^{r_d-r_1}&&\\
&\ddots&\\
&&1
\end{pmatrix}\cdot A_0,
\]
since $\vp(t)=pt$.

  By the proof of \cite[Proposition~V.2.3]{berger04}, we can lift it to a $\AQp$-basis of $\NN(\T(-r_d))$ (say $n_1',\ldots,n_d'$) such that the matrix of $\vp$ with respect to this basis is given by
  \[
  \begin{pmatrix}
(\mu q)^{r_d-r_1}&&\\
&\ddots&\\
&&1
\end{pmatrix}\cdot A_0.
  \]
  But $n_1'\otimes \pi^{-r_d}e_{r_d},\ldots, n_d'\otimes \pi^{-r_d}e_{r_d}$ is a basis of $\NN(\T)$ and $$\vp(n_i'\otimes \pi^{-r_d}e_{r_d})=q^{-r_d}\vp(n_i')\otimes \pi^{-r_d}e_{r_d},\quad i=1,\ldots,d.$$ Hence the result.
\end{proof}

Note in particular that the inverse $P^{-1}$ is a matrix with coefficients in $\AQp$. Let $n_1,\ldots,n_d$ be the basis given by Proposition~\ref{prop:Wachbasis} (so $n_i\mod \pi=v_i$ for $i=1,\ldots,d$). We have a change of basis matrix $M=(m_{ij})\in M_{d\times d}(\Brig)$ that satisfies  $n_i=\sum_{j=1}^dm_{ij}v_j$ for $i=1,\ldots,d$. By the semi-linearity of the action of $\vp$ on $\NN(T)$, we have the equation 
\begin{equation}
PM=\vp(M)A,\label{eq:actions}
\end{equation}
which is equivalent to $M=P^{-1}\vp(M)A$. Substituting $M$ on the right-hand side repeatedly, we obtain for any $n\ge0$ the relation
\begin{equation}\label{eq:phirelation}
M=P^{-1}\vp(P^{-1})\cdots\vp^n(P^{-1})\vp^{n+1}(M)\vp^n(A)\cdots \vp(A)A.
\end{equation}

It follows from \cite[Proposition III.2.1]{berger04} that there is an inclusion of $\Brig$-modules
\begin{equation}\label{eq:changeWach}
\Brig\otimes\Dcris(\T(-r_d))\subset\Brig\otimes\NN(\T(-r_d))
\end{equation}
with elementary divisors $[(t/\pi)^{r_1};\ldots;(t/\pi)^{r_d}]$. Therefore, via \eqref{eq:twistDieudonne} and \eqref{eq:twistWach}, we  deduce that
\[
\Brig\otimes\NN(\T)\subset\Brig\otimes\Dcris(\T)
\]
whose elementary divisors are $[(t/\pi)^{r_1};\ldots;(t/\pi)^{r_d}]$.

\begin{prop}\label{prop:matrixcong}
The matrix $M$ lies inside $I_d+\pi^{r_d} M_{d\times d}(\Brig)$, where $I_d$ is the $d\times d$ identity matrix.
\end{prop}
\begin{proof}
Let $n_1',\ldots,n_d'$ be the basis in the proof of Proposition~\ref{prop:Wachbasis}. Let $G_\gamma$ be the matrix of $\gamma$ with respect to this basis. Then, the proof of \cite[Proposition~V.2.3]{berger04} tells us that
\[
G_\gamma\in I_d+\pi^{p-1}M_{d\times d}(\AQp).
\]

Let $M'=(m_{ij}')=(t/\pi)^{r_d}M^{-1}$, which is the change of basis matrix for the inclusion 
\eqref{eq:changeWach}. In particular, this is a matrix with entries in $\Brig$. We have 
\[
v_i\otimes t^{r_d} e_{-r_d}=\sum_{j=1}^{d}m_{ij}'n_j',\quad i=1,\ldots, d.
\]
The left-hand side of this equation above is invariant under the action of $\gamma$, this implies the relation
\[
\gamma(M')G_\gamma=M'.
\]

By definition, $M'\equiv I_d\mod \pi$. Let $M'=I_d+\pi N$, where $N$ is a matrix defined over $\Brig$. Then, the fact that $G_\gamma\equiv I_d\mod \pi^{p-1}$ implies that
\[
\gamma(\pi)\gamma(N)\equiv \pi N\mod\pi^{p-1}.
\]
On comparing the coefficients of $\pi$, we see that $N\equiv 0\mod \pi$. In other words, $M'\equiv I_d\mod \pi^2$. On repeating this procedure, we can show that $M'\equiv I_d\mod \pi^{p-1}$.

But $M=(t/\pi)^{r_d}(M')^{-1}$ and $t/\pi\equiv 1\mod \pi$. Since we have assumed that $r_d\le p-1$ (i.e. (H.FL)), we have  $M\equiv I_d\mod \pi^{r_d}$ as required.
\end{proof}

\begin{defn}For an integer  $n\ge0$, we define
\[
M_{\log,n}=\fM_F^{-1}\left((1+\pi)\vp^n(M)\right).
\]
\end{defn}

\begin{corollary}\label{cor:logmellin}
For all $n\ge1$,
\[
\omega_{n-1,\{0,\ldots, r_d-1\}}(\gamma)|\left(M_{\log,n}-I_d\right)
\]
over $\calH_F$.
\end{corollary}
\begin{proof}
By Proposition~\ref{prop:matrixcong}, we have
\[
(1+\pi)\vp^n(M)\in (1+\pi)(I_d+\vp^n(\pi)^{r_d} M_{d\times d}(\Brig)).
\]
Therefore, our result follows from Corollary~\ref{cor:mellin}.
\end{proof}

\subsection{Integrality conditions on crystalline classes}
We shall now make use of the structure of $\NN(\T)$ to study crystalline classes (that is, elements in $H^1_f$). 
For any $z\in\HIw(F,\T)$, we define $ \fc_i(z),\fc_{i,n}(z)\in\calH_F $, $i=1,\ldots,d$ by the relations
\begin{equation}
\calL_{\T}(z)=\sum_{i=1}^d\fc_i(z)\otimes v_i= \sum_{i=1}^d \fc_{i,n}(z)\otimes\vp^n(v_i),
\label{eq:expansion}
\end{equation}
where $v_1,\ldots,v_d$ is a basis of $\Dcris(\T)$ as in the previous section and $\calL_\T$ is the regulator map defined in \eqref{eq:regulator}.
Our main goal in this section is to study the coefficients $\fc_{i,n}(z)$ when $z$ comes from classes in $H^1_f(F_n,T)$. Let us first prove two lemmas.







\begin{lem}
 Let $x\in\NN(\T)^{\psi=1}$, then there exist unique elements $s_i(x)\in(\AQp)^{\psi=0}$ for $i=1,\ldots, d$ such that\begin{equation}\label{eq:1-phi}
(1-\vp)x=\sum_{i=1}^d s_i(x)\vp(n_i)=\begin{pmatrix}
s_1(x)&\cdots& s_d(x)
\end{pmatrix}PM\begin{pmatrix}
v_1\\ \vdots\\ v_d
\end{pmatrix},
\end{equation}
where $M$ and $P$ are the matrices as defined in the previous section.
\end{lem}
\begin{proof}
Recall from \cite[\S3]{leiloefflerzerbes11} that we have the inclusion
\[
\NN(\T)^{\psi=1}\stackrel{1-\vp}{\longrightarrow}((\vp)^*\NN(\T))^{\psi=0}\subset(\Brig)^{\psi=0}\otimes\Dcris(\T).
\]
Here $(\vp)^*\NN(\T)$ denotes $\AQp\cdot \vp(\NN(\T))$.
Therefore, we may write \[(1-\vp)x=\sum_{i=1}^d s_i(x)\vp(n_i)\]
for some $s_i(x)\in(\AQp)^{\psi=0}$. Recall that $n_i=\sum_{j=1}^dm_{ij}v_j$. Hence, on applying $\vp$, we obtain
\[
\begin{pmatrix}
\vp(n_1)\\\vdots\\\vp(n_d)
\end{pmatrix}
=\vp(M)\begin{pmatrix}
\vp(v_1)\\\vdots\\\vp(v_d)
\end{pmatrix}=\vp(M)A\begin{pmatrix}
v_1\\\vdots\\v_d
\end{pmatrix},
\]
so we are done by \eqref{eq:actions}.
\end{proof}

\begin{lem}\label{lem:evaluation2}
Let $m\ge0$ and $n\ge1$ be integers,  $x\in(\Brig\otimes\Dcris(\T))^{\psi=1}$ and write $x=\sum_{i=1}^dx_iv_i$, where $x_i\in \Brig$. Then,
\[
\partial_{\T(-m)}\circ\vp^{-n}\left(x\cdot  e_{-m}\right)=\frac{1}{m!}\sum_{i=1}^d(\partial^mx_i)(\zeta_{p^n}-1)\otimes\vp^{-n}(v_i\cdot t^me_{-m}).
\]
\end{lem}
\begin{proof}
Let us write
\[
x\cdot e_{-m}=\sum_{i=1}^d(t^{-m}x_i)\otimes(v_i\cdot t^me_{-m})\in(\Brig[t^{-1}]\otimes\Dcris(\T(-m)))^{\psi=1}.
\]
If we apply $\vp^{-n}$ on both sides, we deduce that
\[
\vp^{-n}\left(x\cdot e_{-m}\right)=\sum_{i=1}^d\left(t^{-m}x_i(\zeta_{p^n}\exp(t/p^n)-1)\right)\otimes\left(\vp^{-n}(v_i)\cdot t^me_{-m}\right)
\]
as $\vp^{-n}(t^{-m})=p^{-nm}t^{-m}$ and $\vp^{-n}(t^m\cdot e_{-m})=p^{nm}t^m\cdot e_{-m}$.
Therefore, when we apply the map $\partial_{\T(-m)}$, what we get on the right-hand side are the coefficients of $t^m$ in $$x_i(\zeta_{p^n}\exp(t/p^n)-1)=\vp^{-n}(x_i), $$
$i=1,\ldots,d$, tensored by $\vp^{-n}(v_i)\cdot t^me_{-m}$.
 The said coefficients in $\vp^{-n}(x_i)$ can be obtained from applying $\frac{1}{m!}\frac{d^m}{dt^m}|_{t=0}=\frac{1}{m!}\partial^m|_{t=0}$. However, we in fact have the relation $p^{mn}\partial^m\vp^{-n}=\vp^{-n}\partial^m$, which implies that
\[
\frac{1}{m!}\partial^m\vp^{-n}(x_i)=\frac{1}{p^{mn}m!}\vp^{-n}\partial^m(x_i).
\]
In other words,
\begin{align*}
\partial_{\T(-m)}\circ\vp^{-n}\left(x\cdot e_{-m}\right)&=\frac{1}{p^{mn}m!}\sum_{i=1}^d\vp^{-n}\partial^m(x_i)|_{t=0}\otimes\left(\vp^{-n}(v_i)\cdot t^me_{-m}\right)\\
&=\frac{1}{m!}\sum_{i=1}^d\vp^{-n}\partial^m(x_i)|_{t=0}\otimes\vp^{-n}(v_i\cdot t^me_{-m}).
\end{align*}
Given an element  $R\in\Brig$, we have $\vp^{-n}(R)|_{t=0}=R(\zeta_{p^n}-1)$. Hence the result.
\end{proof}

Using the second part of Theorem~\ref{thm:berger}, we deduce similarly the following lemma.

\begin{lem}\label{lem:evaluation2*}
Let $m\ge0$ be an integer,  $x=\sum_{i=1}^dx_iv_i\in(\Brig\otimes\Dcris(\T))^{\psi=1}$be as in Lemma~\ref{lem:evaluation2}. Then,
\[
\partial_{\T(-m)}\circ(1-p^{-1}\vp^{-1})\left(x\cdot  e_{-m}\right)=\frac{1}{m!}\sum_{i=1}^d(\partial^mx_i)(0)\otimes(1-p^{m-1}\vp^{-1})(v_i\cdot t^me_{-m}).
\]
\end{lem}

\begin{theorem}\label{thm:PRdivisibility}
Let $0\le m\le r_d-1$ and $n\ge1$ be integers. Fix an element $z\in\HIw(F,\T)$ such that $\pr_{\T(-m),n}(z\cdot e_{-m})\in H^1_f(F_n,\T(-m))$. If $\fc_{i,n}(z)\in\calH_F$ are the elements defined in \eqref{eq:expansion}, then for $i=1,\ldots,d$,
\[
\fc_{i,n}(z)=\omega_{n-1,m}(\gamma)\tfc_{i,n}(z)
\]
for some $\tfc_i^n(z)\in\calH_F$. Furthermore, 
\[
\tfc_{i,n}(z)\equiv \tilde{\ft}_i\mod\omega_{n-1,\{0,\ldots,r_d-1\}}(\gamma)\calH_F
\]
for some $ \tilde{\ft}_i\in\Lambda_F$.
\end{theorem}
\begin{proof}
Let $x=(h^1_\T)^{-1}(z)\in\NN(\T)^{\psi=1}$. Then \eqref{eq:PRregulator} says that
\[
\calL_{\T}(z)=(\fM^{-1}\otimes1)\circ(1-\vp)x.
\]

Write $x=\sum_{i=1}^dx_iv_i$ and $(1-\vp)x=\sum_{i=1}^du_{i,n}\vp^n(v_i)$ where $x_i\in\Brig$ and $u_{i,n}\in(\Brig)^{\psi=0}$. On comparing with \eqref{eq:expansion}, we have
\begin{equation}
\fM_F\left(\fc_{i,n}(z)\right)=u_{i,n}.\label{eq:transform}
\end{equation} Furthermore, we may deduce from \eqref{eq:phirelation} and \eqref{eq:1-phi} that
\begin{equation}\label{eq:relations}
\begin{split}
\begin{pmatrix}
u_{1,n}&\cdots&u_{d,n}
\end{pmatrix}=&
\begin{pmatrix}
s_1(x)&\cdots&s_d(x)
\end{pmatrix}PM(\vp^{n-1}(A)\cdots\vp(A)A)^{-1}\\
=&\begin{pmatrix}
s_1(x)&\cdots&s_d(x)
\end{pmatrix}\vp(P^{-1})\cdots \vp^{n-1}(P^{-1})\vp^{n}(M).
\end{split}
\end{equation}

By Theorem~\ref{thm:berger} and Lemma~\ref{lem:evaluation2}, the fact that $\pr_{\T(-m),n}(z\cdot e_{-m})\in H^1_f(F_n,\T(-m))$ implies that  $$(\partial^mx_i)(\zeta_{p^n}-1)=0.$$ Therefore, $\vp^{n-1}(q)|\partial^mx_i$ (over $\Brig$) for all $i$. 

Further, note that if $\pr_{\T(-m),n}(z\cdot e_{-m})\in H^1_f(F_n,\T(-m))$, then $\pr_{\T(-m),k}(z\cdot e_{-m})\in H^1_f(F_k,\T(-m))$ for all $0\le k\le n$.  On applying Lemma~\ref{lem:evaluation2} again, we have $\vp^{k-1}(q)|\partial^mx_i$ for $1\le k\le n$. By Lemma~\ref{lem:evaluation2*} and (H.eigen), we also have $\pi|\partial^mx_i$. Therefore, we deduce that $\vp^n(\pi)|\partial^mx_i$ for all $i$. On applying $1-\vp$, we see that $\vp^n(\pi)|\partial^mu_{i,n}$ for all $i$. The first part of the theorem now follows from \eqref{eq:transform} and Corollary~\ref{cor:mellin}.

Consider \eqref{eq:relations}. Let $\tau_1,\ldots, \tau_d\in(\AQp)^{\psi=0}$ be the elements defined by
\[
\begin{pmatrix}
\tau_1&\cdots&\tau_d
\end{pmatrix}=
\begin{pmatrix}
s_1(x)&\cdots&s_d(x)
\end{pmatrix}\vp(P^{-1})\cdots \vp^{n-1}(P^{-1}).
\]
We then have
\[
\begin{pmatrix}
\tau_1&\cdots&\tau_d
\end{pmatrix}\vp^n(M)=\begin{pmatrix}
u_{1,n}&\cdots&u_{d,n}
\end{pmatrix}.
\]
Recall from \cite[proof of Theorem~3.5]{leiloefflerzerbes10} that if $f\in\Brig$, then
\[
(\AQp\cdot\vp(f))^{\psi=0}
\]
is a free $\Lambda_F$-module of rank 1, with basis $(1+\pi)\vp(f)$. Therefore, there exist $\ft_i\in\Lambda_F$ such that
\[
\begin{pmatrix}
\ft_1&\cdots&\ft_d
\end{pmatrix}\cdot(1+\pi)\vp^n(M)=
\begin{pmatrix}
\tau_1&\cdots&\tau_d
\end{pmatrix}\vp^n(M)=\begin{pmatrix}
u_{1,n}&\cdots&u_{d,n}
\end{pmatrix}.
\]
We deduce from \eqref{eq:transform} that
\begin{equation}
\begin{pmatrix}
\ft_1&\cdots&\ft_d
\end{pmatrix}M_{\log,n}=\begin{pmatrix}
\fc_{1,n}(z)&\cdots \fc_{d,n}(z)
\end{pmatrix}.\label{eq:matrixequation}
\end{equation}
But $M_{\log,n}\equiv I_d\mod \omega_{n-1,\{0,\ldots, r_d-1\}}$ by Corollary~\ref{cor:logmellin}. Hence, the fact that $\omega_{n-1,m}(\gamma)|\fc_{i,n}(z)$ (over $\calH_F$) implies that $\omega_{n-1,m}(\gamma)|\ft_i$ over $\calH_F$ for all $i$. Since $\ft_i\in\Lambda_F$ and $\omega_{n-1,m}(\gamma)$ is a monic polynomial in $\gamma-1$, this division is in fact over $\Lambda_F$. If we write $\tilde{\ft}_i=\ft_i/\omega_{n-1,m}(\gamma)\in\Lambda_F$, then \eqref{eq:matrixequation} tells us that
\[
\begin{pmatrix}
\tilde{\ft}_1&\cdots&\tilde{\ft}_d
\end{pmatrix}M_{\log,n}=\begin{pmatrix}
\tfc_{1,n}(z)&\cdots \tfc_{d,n}(z)
\end{pmatrix}.
\]
Hence,
\[
\tilde{\ft}_i\equiv\tfc_{i,n}(z)\mod\omega_{n-1,\{0,\ldots,r_d-1\}}(\gamma)\calH_F
\]
as required.
\end{proof}

Let $S_{n,m}$ be the set of elements in $\calH_F$ that are coprime to $\omega_{n-1,m}(\gamma)$. If $M$ is a $\calH_F$-module, we write $S_{n,m}^{-1}M$ for the localization of $M$ with respect to $S_{n,m}$.

\begin{corollary}\label{cor:regulator}
Let $m$ and $z$ be as above.  Then,
\[\frac{\calL_{\T}(z)}{\prod_{j=0}^{r_d-1}\ell_j}\in S^{-1}_{n,m}\calH_F\otimes\Dcris(\T).\]
Furthermore, there exist $\fs_{i,m}\in\Lambda_F$, $i=1,\ldots, d$ such that
\[
\frac{\calL_{T}(z)}{\prod_{j=0}^{r_d-1}\ell_j}\equiv p^{n}\sum_{i=1}^d\fs_{i,m}\otimes\vp^n(v_i)\mod\omega_{n-1,m}(\gamma)S_{n,m}^{-1} \calH_F\otimes\Dcris(\T).
\] 
\end{corollary}
\begin{proof}
We use the same notation as in the proof of Theorem~\ref{thm:PRdivisibility}. We have
\begin{align*}
\frac{\calL_{\T}(z)}{\prod_{j=0}^{r_d-1}\ell_j}&=\frac{1}{\ell_{r_d-1}\cdots\ell_0}\sum_{i=1}^d\fc_{i,n}(z)\otimes\vp^n(v_i)\\
&=\frac{\omega_{n-1,m}(\gamma)}{\ell_{r_d-1}\cdots\ell_0}\sum_{i=1}^d\tfc_{i,n}(z)\otimes\vp^n(v_i),
\end{align*}
which gives the first part of the corollary as $\ell_j$ with $j\ne m$ and $\ell_m/\omega_{n-1,m}(\gamma)$ are all coprime to $\omega_{n-1,m}(\gamma)$.

Now, consider the factor
\[
\frac{\omega_{n-1,m}(\gamma)}{\ell_{r_d-1}\cdots\ell_0}=\frac{p^{(n-1)}\log(u)}{\prod_{i\ge n}\frac{\Phi_{i,m}(\gamma)}{p}}\prod_{j\ne m}\frac{1}{\ell_j}.
\]
On the one hand, we deduce from Lemma~\ref{lem:modulounit} and (H.FL) that $\ell_j$ is congruent to $(m-j)$ modulo $\omega_{n-1,m}(\gamma)$, which is a $p$-adic unit  for all $j\ne m$. On the other hand, Lemma~\ref{lem:modulo1} says that $\prod_{i\ge n}\frac{\Phi_{i,m}(\gamma)}{p}$ is  congruent to a constant in $1+p^n\Zp$. Hence we are done by the following fact
\[
v_p(\log(u))=v_p(u-1)=1.
\]
\end{proof}

\begin{corollary}\label{cor:inversePR}
Let $m$ and $z$ be as above. There exist elements $c_{i,m}\in(\AQp)^{\psi=0}$, $i=1,\ldots, d$ such that
\[
p^{ n}\sum_{i=1}^d\partial^mc_{i,m}\otimes\vp^n(v_i)\equiv\partial^m\circ\Omega_{\T,r_d}^{-1}(z)\mod\vp^n(\pi)S_{n,0}^{-1}(\Brig)^{\psi=0}\otimes\Dcris(\T).
\]
\end{corollary}
\begin{proof}
Let $\fM_{F,S_{n,m}}$ be the induced map from $S_{n,m}^{-1}\calH_F$ to $S_{n,m}^{-1}(\Brig)^{\psi=0}$. Then, by Corollary~\ref{cor:mellin}, we have
\[
\partial^m\fM_{F,S_{n,m}}\left(\omega_{n-1,m}(\gamma)S^{-1}_{n,m} \calH_F\right)=\vp^n(\pi)S_{n,0}^{-1}(\Brig)^{\psi=0}.
\]
Our result now follows from combining this with Corollary~\ref{cor:regulator} and \eqref{eq:regulator}.
\end{proof}

\begin{prop}\label{prop:inversePR}Let $c_{i,m}\in(\AQp)^{\psi=0}$ be any elements given by Corollary~\ref{cor:inversePR}. Then
\[
\pr_{\T(-m),n}(z)=(-1)^m(r_d-m-1)!p^{(m+1-r_d)n}\Sigma_{\T(-m),n}\circ\Tw_{-m}\left(\sum_{i=1}^dc_{i,m}\otimes v_i\right).
\]
\end{prop}
\begin{proof}
Let us write 
\[
c=\Tw_{-m}\left(\sum_{i=1}^dc_{i,m}\otimes v_i\right)\in(\AQp)^{\psi=0}\otimes\Dcris(\T(-m))\]
 and let
 \[
 c_n=p^{ n}\sum_{i=1}^d\partial^mc_{i,m}\otimes \vp^n(v_i)=(p\otimes \vp)^n(c)\cdot t^{-m}e_m.
 \]
  By the commutative diagram \eqref{eq:PRcommute}, we have
 \[
\Omega_{\T(-m),r_d-m}^{-1}(z\cdot e_{-m})=(-1)^m \Tw_{-m}\circ\Omega_{\T,r_d}^{-1}(z)=(-1)^m (\partial^m\otimes t^me_{-m})\circ\Omega_{\T,r_d}^{-1}(z).
\] 
Corollary~\ref{cor:inversePR} says that the last term is congruent to $$c_n\cdot t^me_{-m}=(p\otimes\vp)^n(c)\mod \vp^n(\pi).$$
Therefore, on applying \eqref{eq:PRXi}, we have
\[
\pr_{\T(-m),n}(z\cdot e_{-m})=(-1)^m(r_d-m-1)!\Xi_{\T(-m),n}\left((p\otimes\vp)^n(c)\right).
\]

By definitions,
\[
p^{(r_d-m-1)n}\Xi_{\T(-m),n}=\Sigma_{\T(-m),n}\circ (p\otimes \vp)^{-n}
\]
as the largest Hodge-Tate weight of $\T(-m)$ is $r_d-m$. Hence, we deduce that
\begin{align*}
\pr_{\T(-m),n}(z\cdot e_{-m})&=(-1)^m(r_d-m-1)!p^{-(r_d-m-1)n}\Sigma_{\T(-m),n}(c)
\end{align*}
  as required.
\end{proof}

\section{Upper bounds of Tamagawa numbers}

\subsection{Relation between $(\AQp)^{\psi=0}$ and $\calO_{F_n}$}

For $n\ge2$, we let $\calO_{F_n}^{\Tr=0}$ denote the kernel of $\Tr_{F_n/F_{n-1}}$ on $\calO_{F_n}$. By  \cite[Lemma~3.1]{lei10}, we have the decomposition
\begin{equation}
\calO_{F_n}=\bigoplus_{i=2}^n\calO_{F_i}^{\Tr=0}\oplus\calO_{F_1}
\label{eq:decomposition}
\end{equation}
as $\Zp$-modules and that 
\begin{equation}
\rank_{\Zp}\calO_{F_n}^{\Tr=0}=\phi(p^n)-\phi(p^{n-1})=p^{n-2}(p-1)^2.
\label{eq:tracekernelrank}
\end{equation}

\begin{lem}\label{lem:ker}
We have
\[
(\vp^{n-1}(q)\AQp)\cap(\AQp)^{\psi=0}
=
\begin{cases}
\vp^{n-1}(q)(\AQp)^{\psi=0}&\text{if $n\ge2$,}
\\
\vp(\pi)(\AQp)^{\psi=0}&\text{if $n=1$.}
\end{cases}
\]
\end{lem}
\begin{proof}
If $n\ge 2$, $\vp^{n-1}(q)f\in(\AQp)^{\psi=0}$ for some $f\in\AQp$, then
\[\psi(\vp^{n-1}(q)f)=\vp^{n-2}(q)\psi(f)=0.\]
This forces $f\in(\AQp)^{\psi=0}$.

For $n=1$, suppose that $\psi(qf)=0$, where $f\in\AQp$. We may write $f=f(0)+\pi g$ for some $g\in\AQp$. Then $\psi(qf(0)+q\pi g)=f(0)+\pi\psi(g)=0$, so $f(0)=\psi(g)=0$. This implies that $qf\in\vp(\pi)(\AQp)^{\psi=0}$ as claimed.
\end{proof}

\begin{lem}\label{lem:eval}
There is an isomorphism of $\Zp$-modules
\begin{align*}
\iota_n:(\AQp)^{\psi=0}/\vp^{n-1}(q)(\AQp)^{\psi=0}\cong \calO_{F_n}^{\Tr=0}.
\end{align*}
for $n\ge2$. When $n=1$, we have instead
\begin{align*}
{\iota}_1:(\AQp)^{\psi=0}/\vp(\pi)(\AQp)^{\psi=0}\cong \calO_{F_1}.
\end{align*}
\end{lem}
\begin{proof}
Since $\Phi_n$ is the minimal polynomial of $\zeta_{p^n}-1$ and $\vp^{n-1}(q)=\Phi_n(\pi+1)$ for $n\ge1$, we have an isomorphism of rings
\begin{equation}
\AQp/\vp^{n-1}(q)\AQp\cong\calO_{F_n}
\label{eq:isorings}
\end{equation}
given by the evaluation map $\iota_n:\pi\mapsto\zeta_{p^n}-1$. By Lemma~\ref{lem:ker}, the map $\iota_n$ restricted to $(\AQp)^{\psi=0}$ has kernel 
$\vp^{n-1}(q)(\AQp)^{\psi=0}$ when $n\ge2$, whereas for  $n=1$, the kernel is $\vp(\pi)(\AQp)^{\psi=0}$. Therefore, we have injections
\begin{align*}
\iota_n:(\AQp)^{\psi=0}/\vp^{n-1}(q)(\AQp)^{\psi=0}\hookrightarrow& \calO_{F_n},\\
{\iota}_1:(\AQp)^{\psi=0}/\vp(\pi)(\AQp)^{\psi=0}\hookrightarrow& \calO_{F_1}.
\end{align*}

Given an element $f\in(\AQp)^{\psi=0}$, we have
\[
\sum_{\eta^p=1}f(\eta(1+\pi)-1)=0.
\] 
Therefore, for $n\ge2$,
\[
\Tr_{F_n/F_{n-1}}f(\zeta_{p^n}-1)=0.
\]
Whereas,
\[
\Tr_{F_1/F_{0}}f(\zeta_{p}-1)=-f(0).
\]
This implies that $\iota_n(f)\in\calO_{F_n}^{\Tr=0}$ if $n\ge2$.  If $x\in \calO_{F_n}^{\Tr=0}$, then 
\[
x=\sum_{1\le i\le p^n-1,(i,p)=1}a_i\zeta_{p^n}^i
\]
for some $a_i\in\calO_F$. So, $x$ has a pre-image
\[
\sum_{1\le i\le p^n-1,(i,p)=1}a_i(1+\pi)^i\in(\AQp)^{\psi=0}
\]
under the map $\iota_n$. This gives our result for $n\ge2$.

Now consider $n=1$. Given $x\in\calO_{F_1}$, we have
\[
x=\sum_{i=1}^{p-1}a_i\zeta_{p}^i
\]
for some $a_i\in\calO_F$. So, if
\[
f=\sum_{i=1}^{p-1}a_i(1+\pi)^i,
\]
then $\iota_1(f)=x$. Hence we are done.
\end{proof}

\subsection{Bounding Tamagawa numbers}\label{S:tamagawa}


Now, we assume that $T$ is a crystalline representation satisfying (H.eigen) and (H.FL) with Hodge-Tate weights in $[a,b]$, where $a\le 0$ and $b\ge1$ as before. Let $1\le r\le b$ be an integer such that the slopes of $\vp$ on $\Dcris(T)$ are all $\ge -r$. We fix two integers $s_1,s_2\ge0$ such that
\begin{align}
(p^{r}\vp)^k\Dcris(T)\subset\ & p^{-s_1}\Dcris(T)\ \forall k\ge0,\label{eq:neweigen}\\
(1-\vp)^{-1}\Dcris(T)\subset\ & p^{-s_2}\Dcris(T).\label{eq:newnewlattice}
\end{align}

 Let us first recall the definition of the Tamagawa number of $T$ over $F_n$ with respect to a basis $\omega$ of $\det_{\Qp}t_n(V)$. Let $\omega^*$ be the dual basis to $\omega$. Then, the Tamagawa number  is the unique power of $p$ defined by
\[
\iota(\det{}_{\Zp}^{-1}H^1_f(K,T))=\Zp\Tam_{K,\omega}(T)\omega^*,
\]
where $\iota$ is the isomorphism $\det{}_{\Qp}^{-1}H^1_f(F_n,V)\rightarrow\det_{\Qp}^{-1}t_n(K)$, which arises from the exact sequence
\[
0\rightarrow \Dcris(V)\rightarrow \Dcris(V)\oplus t_n(V)\rightarrow H^1_f(F_n,V)\rightarrow0
\]
which is a consequence of \eqref{eq:exact} and (H.eigen). Hence, if we fix a lattice $L$ in $t_n(V)$ such that $\omega$ generates $\det_{\Zp}L$, then 
\begin{equation}\label{eq:Tam}
\Tam_{F_n,\omega}(T)=\frac{\left(H^1_f(F_n,T):\exp_{V,n}(L)\right)}{\det_{\Zp}(1-\vp|\Dcris(V))_{(p)}}.
\end{equation}

From now on, we take $L_n$ to be the lattice $\calO_{F_n}\otimes \Dcris(T)/\Fil^0\Dcris(T)$ and fix a basis $\omega_n$ of $\det_{\Zp}L_n$. We write $d'$ for the $\Zp$-rank of $\Dcris(T)/\Fil^0\Dcris(T)$ and $d'_n:=d'\times[F_n:F]$, which is the $\Zp$-rank of $\calO_{F_n}\otimes\Dcris(T)/\Fil^0\Dcris(T)$. 

\begin{lem}\label{lem:scalarmultiplication}
For $m\in\ZZ$, let $L_{n,m}=p^mL_n$ and $\omega_{n,m}$ a basis of $\det_{\Zp}L_{n,m}$. We have the equality 
\[
\Tam_{F_n,\omega_{n,m}}(T)=p^{md'_n}\Tam_{F_n,\omega_n}(T).
\]
\end{lem}
\begin{proof}
It follows from \eqref{eq:Tam} and the fact that the $\Zp$-rank of $H^1_f(F_n,T)$ is $d'_n$.
\end{proof}

From now on, we shall assume that the representation $T$ satisfies the following additional hypothesis.

\vs

\noindent
\underline{\textbf{(H.tor)}} $(T^\vee)^{G_{F_\infty}}=0$, where $T^\vee$ denotes the Pontryagin dual of $T$.

\vs

Under (H.tor), the inflation-restriction exact sequence implies that the corestriction map $H^1(F_m,T)\rightarrow H^1(F_n,T)$ is surjective for all $m\ge n$ (see for example \cite[Corollary~4.5]{lei10}). In particular, the projection $\HIw(F,T)\rightarrow H^1(F_n,T)$ is surjective.

\begin{lem}\label{lem:inclusion-expo}
Let $x\in H^1_f(F_n,T)$, then there exists $g\in(\AQp)^{\psi=0}\otimes \Dcris(T)$ such that
\[
\exp_{V,n}\circ\Theta_{T,n}(g)=x.
\]
\end{lem}
\begin{proof}
We take $m=-a\ge0$ in Proposition~\ref{prop:inversePR}. In particular, $\T(-m)=T$ by the definition of $\T$. Let $z\in \HIw(F,T)$ be any element such that
\[
\pr_{\T(-m),n}(z)=x.
\]
Proposition~\ref{prop:inversePR} says that there exists $c\in(\AQp)^{\psi=0}\otimes\Dcris(\T)$ such that
\[
x=p^{(1-b)n}\Sigma_{T,n}\circ\Tw_{-m}(c).
\]
But we have $\Tw_{-m}(c)\in(\AQp)^{\psi=0}\otimes\Dcris(T)$ and
\[
\Sigma_{T,n}=p^{(b-1)n}\exp_{V,n}\circ\Theta_{T,n}.
\]
Hence, we obtain our result on taking $g=\Tw_{-m}(c)$.
\end{proof}

\begin{corollary}\label{cor:bound}
Let $A_n$ be the image of $(\AQp)^{\psi=0}\otimes \Dcris(T)$ under $\Theta_{T,n}$ in $t_n(V)$. We have the inequality
\[
\left(H^1_f(F_n,T):\exp_{V,n}(L_n)\right)\leq (A_n:L_n)
\]
\end{corollary}
\begin{proof}
By Lemma~\ref{lem:inclusion-expo}, we have the inclusion
\[
H^{1}_f(F_n,T)\subset \exp_{V,n}\left(A_n\right).
\]
Therefore,
\[
\left(H^1_f(F_n,T):\exp_{V,n}(L_n)\right)\leq\left(\exp_{V,n}(A_n):\exp_{V,n}(L_n)\right).
\]
The result now follows from the fact that $\exp_{V,n}$ is injective on $t_n(V)$.
\end{proof}

Via \eqref{eq:decomposition}, there are decompositions 
\[
\begin{split}
t_n(V)&=\bigoplus_{i=2}^n\left(\calO_{F_i}^{\Tr=0}\otimes t(V)\right)\oplus \left(\calO_{F_1}\otimes t(V)\right),\\
L_n&=\bigoplus_{i=2}^n\left(\calO_{F_i}^{\Tr=0}\otimes \Dcris(T)/\Fil^0\Dcris(T)\right)\oplus \left(\calO_{F_1}\otimes \Dcris(T)/\Fil^0\Dcris(T)\right).
\end{split}
\]
For $i\ge2$ (respectively $i=1$), we write $\pr_i$ for the projection of $t_n(V)$ to
$\calO_{F_i}^{\Tr=0}\otimes t(V)$ (respectively $\calO_{F_1}\otimes t(V)$).

\begin{lem}\label{lem:project}
Let $g\in(\AQp)^{\psi=0}\otimes\Dcris(T)$, then
\[
\pr_i\circ\Theta_{T,n}(g)=\begin{cases}
\left((\sigma^{-i}\otimes \vp^{n-i})g\right)(\zeta_{p^i}-1)\mod \Fil^0\Dcris(T)&\text{if }i\ge2,\\
\left((\sigma^{-1}\otimes \vp^{n-1})g\right)(\zeta_{p}-1)+(1-\vp)^{-1}\vp^n(g(0))\mod\Fil^0\Dcris(T)&\text{if }i=1.\end{cases}
\]
\end{lem}
\begin{proof}
Recall from \eqref{eq:solution} that if $(1-\vp)G=g$, then
\[
G=\sum_{k=0}^{\infty}\vp^k(\tilde{g})+\sum_{j=0}^r\frac{t^j}{j!}\otimes (1-p^j\vp)^{-1}\Delta_j(g),
\]
where $\tilde{g}=g-\sum_{j=0}^r\frac{t^j}{j!}\otimes\Delta_j(g)$. When evaluated at $\zeta_{p^n}-1$, all the terms involving $t^j$ with $j\ge1$ vanish by the definition of $t$. Furthermore, by the definition of $\tilde{g}$, it is divisible by  $\pi$ over $\Brig$. This implies that $\vp^k(\tilde{g})$ vanishes at $\zeta_{p^n}-1$ for all $k\ge n$. Therefore, we have
\begin{align*}
G^{\sigma^{-n}}(\zeta_{p^n}-1)&=\sum_{k=0}^{n-1}\vp^k(\tilde{g})^{\sigma^{-n}}(\zeta_{p^n}-1)+(1-\vp)^{-1}g(0)\\
&=\sum_{k=0}^{n-1}\left((\sigma^{k-n}\otimes \vp^k)g(\zeta_{p^{n-k}}-1)-\vp^k(g(0))\right)+ (1-\vp)^{-1}g(0)\\
&=\sum_{i=1}^n \left((\sigma^{-i}\otimes \vp^{n-i})g\right)(\zeta_{p^{i}}-1)+(1-\vp)^{-1}\vp^n(g(0)).
\end{align*}
Our result now follows from Lemma~\ref{lem:eval}.
\end{proof}

\begin{prop}\label{prop:bound}
We have the inequality
\[
(A_n:L_n)\le p^{(r(p^{n-1}+p-2)+s_2(p-1))d'+s_1d_n'}.
\]
\end{prop}
\begin{proof}
Let $g\in(\AQp)^{\psi=0}\otimes\Dcris(T)$. For $i\ge2$, Lemma~\ref{lem:project} tells us that  $\pr_i\circ \Theta_{T,n}(g)$ lies inside \[
\calO_{F_i}^{\Tr=0}\otimes \vp^{n-i}(\Dcris(T))\mod\Fil^0\Dcris(T).
\]
Note that $\vp^{n-i}(\Dcris(T))\subset p^{-s_1+r(i-n)}\Dcris(T)$ by \eqref{eq:neweigen}, so
\[
\pr_i(A_n)\subset p^{-s_1+r(i-n)}\pr_i(L_n).
\]
Together with \eqref{eq:tracekernelrank} we deduce that
\[
[\pr_i(A_n):\pr_i(L_n)]\le \left(p^{s_1+r(n-i)}\right)^{p^{i-2}(p-1)^2d'}.
\]

For $i=1$, \eqref{eq:neweigen} and \eqref{eq:newnewlattice} tell us that
\[
\pr_1(A_n)\subset\calO_{F_1}\otimes p^{-s_1-s_2-rn}\pr_1(L_n).
\] 
Thus,
\[
[\pr_1(A_n):\pr_1(L_n)]\le \left(p^{s_1+s_2+rn}\right)^{(p-1)d'}.
\]

Since $L_n=\bigoplus_{i=1}^n\pr_i(L_n)$ and $A_n\subset\bigoplus_{i=1}^n\pr_i(A_n)$, we deduce that
\begin{align*}
(A_n:L_n)&\le\prod_{i=1}^n(\pr_i(A_n):\pr_i(L_n))\\
&\le \left(p^{\sum_{i=2}^n(s_1+r(n-i))p^{i-2}(p-1)^2+(s_1+s_2+rn)(p-1)}\right)^{d'}\\
&=p^{(r(p^{n-1}+p-2)+s_2(p-1))d'+s_1d_n'}
\end{align*}
as required.
\end{proof}

\begin{corollary}\label{cor:finalbound}
The Tamagawa number $\Tam_{F_n,\omega_n}(T)$ is bounded above by
\[
\frac{p^{(r(p^{n-1}+p-2)+s_2(p-1))d'+s_1d_n'}}{\det_{\Zp}(1-\vp|\Dcris(V))_{(p)}}.
\]
\end{corollary}
\begin{proof}
This follows from combining \eqref{eq:Tam} with Corollary~\ref{cor:bound} and Proposition~\ref{prop:bound}. 
\end{proof}

\begin{remark}
Note that $p^{n-1}+p-2\le [F_n:F]$ for all $n\ge1$. In particular, under the notation of Lemma~\ref{lem:scalarmultiplication},  we deduce from Corollary~\ref{cor:finalbound}  the inequality
\[
\Tam_{F_n,\omega_{n,-r-s_1}}(T)\le \frac{p^{s_2(p-1)d'}}{\det(1-\vp|\Dcris(V))_{(p)}},
\]
which is independent of $n$.
\end{remark}

\section{Comparison with Laurent's results}\label{S:compare}

We recall from \cite{laurent}  that  $\Tam_{F_n,\omega_n}(T)$ is bounded above by
\begin{equation}
p^{[F_n:\Qp]n\alpha}\times j_n(\omega_n\otimes v_n^{-1}),
\label{eq:laurent}
\end{equation}
where $\alpha=-\sum_i^dr_i+(b-1)d+d'$, with  $r_1,\ldots,r_d$ being the Hodge-Tate weights of $T$ and $j_n(\omega_n\otimes v_n^{-1})$ is a power of $p$ that measures the difference between $\omega_n$ and $\det_{\Zp}A_n$ (see p.97 of \textit{op. cit.}). In general, it is not clear how to calculate $j_n(\omega_n\otimes v_n^{-1})$ explicitly. However, it seems to us that the bounds given in Corollary~\ref{cor:finalbound} would generally grow slower than \eqref{eq:laurent} because the former are $O(p^n)$, instead of $O(np^n)$.  In the following section, we illustrate this with representations arising from modular forms.

\begin{remark}
There are in fact some extra factors in the upper bound of Laurent. However, under (H.eigen) and (H.FL), they all turn out to be $1$.
\end{remark}

\subsection{Representations attached to modular forms}

Let $f=\sum_{n\ge 1}a_nq^n$ be a modular form of even weight $k$  and nebentypus $\varepsilon$ whose level is coprime to $p$. Let $E$ be the smallest extension of $\Qp$ that contains $a_n$ for all $n$. Let $V_f$ be the Deligne representation restricted to $G_{\Qp}$, which is of dimension $2[E:\Qp]$.  Furthermore, we fix a Galois-stable lattice $T_f$ inside $V_f$. It has Hodge-Tate weights $1-k$ and $0$ (with multiplicity $[E:\Qp]$). In particular, if $p>k$, then (H.FL) holds as well. The action of $\vp$ on $\Dcris(V_f)$ satisfies $\vp^2-a_p\vp+\varepsilon(p)p^{k-1}$. We see that the evenness of $k$ ensures that (H.eigen) holds as well. Finally, if (H.FL) holds, then $p+1\nmid k-1$. This implies that (H.tor) holds by \cite[Lemma~4.4]{lei10}.

Take $T=T_f(k/2)$ and assume that $v_p(a_p)\ge k/2$ with $\varepsilon(p)=1$. Under this assumption, the slope of $\vp$ on $\Dcris(T)$ is constant and equal to $-1/2$.

For this particular representation, Laurent has worked out the exact value of $j_n(\omega_n\otimes v_n^{-1})$ in \cite[\S3.4.3]{laurent} and showed that
\begin{equation}
\Tam_{F_n,\omega_n}(T_f(k/2))\le p^{(n+1/2)e[F_n:\Qp](k-2)},
\label{eq:laurentsbound}
\end{equation}
where $e=[E:\Qp]$.
Furthermore, when $k=2$, equality holds. That is, $\Tam_{F_n,\omega_n}(T_f(k/2))=1$ for all $n\ge1$.

For the same representation, Corollary~\ref{cor:finalbound} says that
\[
\Tam_{F_n,\omega_n}(T_f(k/2))\le \frac{p^{\left((p^{n-1}+p-2)+s_2(p-1)+s_1[F_n:\Qp]\right)e}}{\det_{\Zp}(1-\vp|\Dcris(V))_{(p)}}
\]
as $d'=e$ and we may take $r=1$. In the case $k>2$, we see that this bound is smaller than \eqref{eq:laurentsbound} when $n$ is sufficiently large, no matter what $s_1$ and $s_2$ are.

For the rest of this section, we shall analyse the action of $\vp$ in order to find out what value $s$ we can take. Recall that the action of $\vp$ on $\Dcris(T)$ satisfies the equation
\begin{equation}\label{eq:phi}
\vp^2-\frac{a_p}{p^{\frac{k}{2}}}\vp+\frac{1}{p}=0.
\end{equation}
Therefore, 
\[
\det(1-\vp|\Dcris(V))_{(p)}=\left(1-\frac{a_p}{p^{\frac{k}{2}}}+\frac{1}{p}\right)_{(p)}=\frac{1}{p^e}.
\]

\begin{lem}
For all $n\ge0$, 
\[
(p\vp)^n\Dcris(T)\subset p^{-k/2+1}\Dcris(T).
\]
\end{lem}
\begin{proof}
By (H.FL), the fact that the highest Hodge-Tate weight of $\Dcris(T)$ is $\frac{k}{2}$ implies that we may pick $v\in\Fil^{k/2}\Dcris(T)\setminus\Fil^{k/2+1}\Dcris(T)$ such that $v,p^{k/2}\vp(v)$ form a basis of $\Dcris(T)$. The matrix of $\vp$ with respect to this basis is given by
\[
\begin{pmatrix}
0 &  -p^{k/2-1}\\
p^{-k/2}&a_p/p^{k/2}
\end{pmatrix}=\begin{pmatrix}
-p^{k/2}\beta&-p^{k/2}\alpha\\
1&1
\end{pmatrix}
\begin{pmatrix}
\alpha&0\\
0&\beta
\end{pmatrix}
\begin{pmatrix}
\frac{p^{k/2}}{\alpha-\beta}&\frac{\alpha}{\alpha-\beta}\\
\frac{-p^{k/2}}{\alpha-\beta}&\frac{-\beta}{\alpha-\beta}
\end{pmatrix},
\]
where $\alpha$ and $\beta$ are the roots to $X^2-\frac{a_p}{p^{\frac{k}{2}}}+\frac{1}{p}$, both of which have $p$-adic valuation $-1/2$. Therefore, the matrix of $(p\vp)^n$ is given by
\[
\begin{pmatrix}
-p^{k/2}\beta&-p^{k/2}\alpha\\
1&1
\end{pmatrix}
\begin{pmatrix}
(p\alpha)^n&0\\
0&(p\beta)^n
\end{pmatrix}
\begin{pmatrix}
\frac{1}{p^{k/2}(\alpha-\beta)}&\frac{\alpha}{\alpha-\beta}\\
\frac{-1}{p^{k/2}(\alpha-\beta)}&\frac{-\beta}{\alpha-\beta}
\end{pmatrix}.
\]
But $(\alpha-\beta)^2=\left(\frac{a_p}{p^{k/2}}\right)^2-\frac{4}{p}$, which means that $v_p(\alpha-\beta)=-1/2$. This tells us that all the entries of the product of the three matrices above have $p$-adic valuations at least $n/2-k/2+1/2\ge -k/2+1$ for $n\ge1$. When $n=0$, the product is the identity. Hence the result.
\end{proof}

\begin{lem}
We have 
\[
(1-\vp)^{-1}\Dcris(T)\subset p^{-\frac{k}{2}+1}\Dcris(T).
\]
\end{lem}
\begin{proof}
Following \cite[Lemma~5.6]{leiloefflerzerbes11}, we deduce from \eqref{eq:phi} the equation
\[
(1-\vp)^{-1}=\frac{\vp+1-\frac{a_p}{p^{\frac{k}{2}}}}{1-\frac{a_p}{p^{\frac{k}{2}}}+\frac{1}{p}},
\]
where the denominator has $p$-adic valuation $-1$. The largest Hodge-Tate weight of $T$ is $\frac{k}{2}$, so (H.FL) implies that $\vp\Dcris(T)\subset p^{-\frac{k}{2}}\Dcris(T)$. As $1+\frac{a_p}{p^{\frac{k}{2}}}$ has non-negative $p$-adic valuation, our result follows.
\end{proof}
In particular, we may take $s_1=s_2=\frac{k}{2}-1$ in Corollary~\ref{cor:finalbound} for this particular representation. Hence, we conclude that 
\[
\Tam_{F_n,\omega_n}(T_f(k/2))\le p^{\left(k/2(p-1)(p^{n-1}+1)-p^{n-1}(p-2)\right)e}.
\]

\section*{Acknowledgment}The author has benefited from invaluable and interesting discussions with Laurent Berger, Henri Darmon, Daniel Disegni, Shinichi Kobayashi and Arthur Laurent during the preparation of this paper. He would also like to thank the anonymous referee for having carefully read an earlier version of the manuscript, which led to a number of improvements in the present article.

\appendix


\section{Results on $p$-adic valuations}

\begin{lem}\label{lem:evaluation}
If $\alpha= 1+p\beta$, where $\beta\in\Zp$, $n\ge0$ is an integer, then $v_p(\alpha^{p^n}-1)=n+1+v_p(\beta)$. Furthermore,
\[
\frac{\alpha^{p^n}-1}{p^{n+1}\beta}\equiv \sum_{i=0}^{n-1}\frac{(-p\beta)^i}{i+1}\mod p^n\Zp.
\] 
\end{lem}
\begin{proof}
We have the expansion
\[
\alpha^{p^n}-1=(1+p\beta)^{p^n}-1=\sum_{i=1}^{p^n}\binom{p^n}{i}(p\beta)^i.
\]
We can rewrite the summand as
\[
\binom{p^n}{i}(p\beta)^i=(p^{n+1}\beta)\times\frac{(p^n-1)\times (p^n-2)\times\cdots (p^n-(i-1))}{1\times 2\times \cdots (i-1)}\times\frac{(p\beta)^{i-1}}{i}.
\]As in \cite[proof of Lemma~5.4]{lei09}, we have
\[
(i-1)\times v_p(p\beta)>v_p(i),
\]
which implies that $p^{n+1}\beta$ divides $\alpha^{p^n}-1$ over $\Zp$ and that
\begin{align*}
\frac{\alpha^{p^n}-1}{p^{n+1}\beta}&=\sum_{i=1}^{p^n}\frac{(p^n-1)\times (p^n-2)\times\cdots (p^n-(i-1))}{1\times 2\times \cdots (i-1)}\times\frac{(p\beta)^{i-1}}{i}\\
&\equiv\sum_{i=1}^{p^n}(-1)^{i-1}\times\frac{(p\beta)^{i-1}}{i}\mod p^n.
\end{align*}
Hence we are done.
\end{proof}

\begin{lem}\label{lem:modulo1}
Let $i,j\in\ZZ$, $k\ge n+1$, then 
\[
\frac{\Phi_{p^k}(u^{-i}\gamma)}{p}\equiv \delta\mod \omega_{n,j}(\gamma).
\]
for some $p$-adic unit $\delta\in 1+p^n\Zp$.
\end{lem}
\begin{proof}
Since $\omega_{n,j}=(u^{-j}\gamma)^{p^n}-1$, we have 
\[
\frac{\Phi_{p^k}(u^{-i}\gamma)}{p}=\frac{(u^{-i}\gamma)^{p^{k}}-1}{p((u^{-i}\gamma)^{p^{k-1}}-1)}\equiv \frac{(u^{j-i})^{p^{k}}-1}{p((u^{j-i})^{p^{k-1}}-1)}\mod \omega_{n,j}
\]
as $k-1\ge n$. If $u=1+p\tau$, where $\tau\in\Zp$, then $u^{j-i}=1+p(j-i)\tau+O(p^2)=1+p\beta$ for some $\beta\in\Zp$. Therefore, if we write $\alpha=u^{j-i}$, Lemma~\ref{lem:evaluation} implies that the last expression satisfies
\[
\frac{(1+p\tau)^{(j-i)p^{k}}-1}{p(((1+p\tau)^{(j-i)p^{k-1}}-1)}=\frac{\alpha^{p^{k}}-1}{p(\alpha^{p^{k-1}}-1)}=\frac{(\alpha^{p^{k}}-1)/p^{k+1}\beta}{(\alpha^{p^{k-1}}-1)/p^k\beta}\in 1+p^n\Zp,
\]
as required.
\end{proof}

\begin{lem}\label{lem:modulounit}
Let $i,j\in\ZZ$ and $n\ge1$. Then, $\ell_i$ is congruent to the constant $j-i$ modulo $­\omega_{n,j}$.
\end{lem}
\begin{proof}
By definition,
\[
\ell_i-\ell_j=j-i
\]
and $\omega_{n,j}|\ell_j$. Hence the result.
\end{proof}

\section{$Q$-systems for crystalline representations}\label{S:congruence}

In \cite{kobayashi13,kobayashi14}, Kobayashi proved the $p$-adic Gross-Zagier formula for elliptic curves and weight 2 modular forms with non-ordinary reduction at $p$. One of the key ingredients is the interpolation of Heegner points. More specifically, let $E$ be an elliptic curve with good supersingular reduction at $p$ with $a_p$ the trace of the Frobenius at $p$ on the $p$-adic Tate module.  Write $\fM_n$ for the maximal ideal of $\Qp(\mu_{p^n})$. If $c_n\in \hat{E}(\fM_n)$, $n\ge1$ is a  family of elements such that
\begin{equation}
\Tr_{n+2/n}(c_{n+2})-a_p\Tr_{n+1/n}(c_{n+1})+pc_n=0
\label{eq:tracerelation}
\end{equation}
for all $n\ge1$, Kobayashi showed that there exists a power series $f\in\Zp[[X]]$ such that $f(\zeta_{p^n}-1)=c_n$ for all $n\ge1$ if and only if $c_{n+1}^p\equiv c_n\mod p\calO_{\Qp(\mu_{p^{n+1}})}$. A family of points satisfying \eqref{eq:tracerelation} is called a $Q$-system, where $Q$ is the polynomial $X^2-a_pX+p$. They have been extensively studied by Knospe in \cite{knospe}. The result of Kobayashi mentioned above has been generalized to general formal groups by Ota in \cite{ota14}. In this appendix, we explain how the results of Perrin-Riou on the solutions to ($E_g$) discussed in \S\ref{S:Eg} allows us to construct $Q$-systems for a general crystalline representation.

Let $T$ be a representation as defined in \S\ref{S:crystalline} with Hodge-Tate weights in $[a;b]$ with $b\ge1$. As before, we assume that both (H.eigen) and (H.FL) hold.
\begin{defn}
Let $k\ge1$ be an integer. We say that $Q(X)\in\calO_F[X]$ is a $k$-polynomial for $T$ if it satisfies
\[
Q(p^k\vp)\Dcris(T)=0.
\]
\end{defn}

\begin{defn}
Let $Q=\sum_{i=0}^Na_iX^i\in\calO_F[X]$ be a $k$-polynomial for $T$ for some integer $k$. We define the $Q$-systems of $T$ to be the set of elements
\[
\left\{(c_n)_{n\ge1}:c_n\in H^1_f(F_n,T),\sum_{i=0}^da_i\cor_{n+i/n}\left(c_{n+i}\right)=0\ \forall n\ge1\right\}.
\]
\end{defn}

\begin{lem}\label{lem:bpoly}
A non-trivial $b$-polynomial for $T$ exists.
\end{lem}
\begin{proof}
Since $\Fil^{-b}\Dcris(T)=\Dcris(T)$, we have $\Fil^0\Dcris(T(-b))=\Dcris(T(-b))$. In particular, by Fontaine-Laffaille theory, $\vp(\Dcris(T(-b)))\subset \Dcris(T(-b))$.  We have the polynomial
\[
Q_T(X)=\det(X^\nu-\vp^\nu|\Dcris(T(-b)))\in\calO_F[X],
\]
where $p^\nu$ is the size of the residue field of $F$. Since $\Dcris(T)=\Dcris(T(-b))\otimes t^{-b}e_b$ and $\vp(t^{-b})=p^{-b}t^{-b}$, we have
\[
Q_T(p^b\vp)\Dcris(T)=Q_T(\vp)\Dcris(T(-b))=0.
\]
So, we see that $Q_T$ is a $b$-polynomial for $T$.
\end{proof}

We fix a $b$-polynomial $Q(X)$ of $T$. Note that $Q(X)$ is then automatically a $b-m$ polynomial for $T(-m)$ for all $m\in\{0,1,\ldots, b-1\}$ (this follows from the fact that $\vp(t^m)=p^mt^m$). In addition, we take $s$ to be the integer satisfying \eqref{eq:eigen} and \eqref{eq:newlattice} (for any fixed $r$). 
 We explain how to construct families of $Q$-system of $T(-m)$ using the results of Perrin-Riou that we have reviewed in \S\ref{S:PR}.

\begin{prop}\label{prop:PRuniversal}
Let $Q$ be a $b$-polynomial for $T$, $g\in(\AQp)^{\psi=0}\otimes\Dcris(T)$ and $m$ an integer such that $0\le m\le b-1$, then $\left(p^{r+s-1}\Sigma_{T(-m),n}\circ\Tw_{-m}(g)\right)_{n\ge1}$ form a $Q$-system of $T(-m)$ for all $m\in\{0,1,\ldots, b-1\}$.
\end{prop}
\begin{proof}
Let $c_n=p^{r+s-1}\Sigma_{T(-m),n}\circ\Tw_{-m}(g)$ for $n\ge1$ and write $Q(X)=\sum_{i=0}^{N}a_iX^i$. Note that \[\Tw_{-m}(g)\in (\AQp)^{\psi=0}\otimes\Dcris(T(-m))\] and that $T(-m)$ has Hodge-Tate weights in $[-m;b-m]$. So, by Corollary~\ref{cor:integralexp}, we have
\[
c_n\in H^1_f(F_n,T(-m))
\]
and
\begin{equation} \label{eq:expression}
\sum_{i=0}^{N}a_i\cor_{n+i/n}(c_{n+i})
=p^{r+s-1}\Sigma_{T(-m),n}\circ\left(\sum_{i=0}^{N} a_i (p^{b-m}\otimes \vp)^i\right) \circ \Tw_{-m}(g).
\end{equation}

Since $Q$ is a $(b-m)$-polynomial for $T(-m)$, we have
\[
\sum_{i=0}^Na_i(p^{b-m}\vp)^i=0
\]
on $\Dcris(T(-m))$.
Hence the right-hand side of \eqref{eq:expression} is $0$.
\end{proof}


We now show that it is possible to refine the construction above by taking into account the slope of $\vp$. In particular, we fix an integer $1\le r\le b$ such that the slopes of $\vp$ on $\Dcris(T)$ are $\ge -r$.

\begin{lem}
A non-trivial $r$-polynomial for $T$ exists. 
\end{lem}
\begin{proof}
Since the eigenvalues of $\vp$ on $\Dcris(T(-r))$ are all $p$-integral, the polynomial
\[
\QTr(X)=\det(X^\nu-\vp^\nu|\Dcris(T(-r))),
\]
where $p^\nu$ is the size of the residue field of $F$, is a polynomial defined over $\calO_F$ and hence it is a $r$-polynomial for $T$. 
\end{proof}

Note that if $Q_T(X)=\sum_{i=0}^{\nu d}a_iX^i$ is the polynomial defined in the proof of Lemma~\ref{lem:bpoly}, then it is related to the polynomial $\QTr$ via the relation\[\QTr(X)=\sum_{i=0}^{\nu d}\frac{a_i}{p^{(b-r)(d-i)}}X^i.\]

We remark that as before, if $Q$ is a $r$-polynomial for $T$, then $Q$ is a $(r-m)$-polynomial for $T(-m)$ for all $m\in\{0,1,\ldots, r-1\}$.

\begin{prop}\label{prop:Qsystem}
Let $Q$ be a $r$-polynomial for $T$. If $g\in(\AQp)^{\psi=0}\otimes\Dcris(T)$, then the elements 
\[
p^{r+s-1+(r-b)n}\Sigma_{T(-m),n}\circ\Tw_{-m}(g)
\]
form a $Q$-system for $T(-m)$ for all $0\le m\le r-1$.
\end{prop}
\begin{proof}Corollary~\ref{cor:integralexp} tells us that
\[
p^{r+s-1+(r-b)n}\Sigma_{T(-m),n}\circ\Tw_{-m}(g)\in H^1_f(F_n,T)
\]
for all $n\ge1$. Let $Q(X)=\sum_{i=0}^Na_iX^i$ and let $$R(X)=p^{(b-r)N}Q(p^{r-b}X)=p^{(b-r)N}\sum_{i=0}^Na_ip^{i(r-b)}X^i\in \calO_F[X].$$ Since $Q$ is a $r$-polynomial for $T$, $R$ is a $b$-polynomial for $T$. Let $c_n=p^{r+s-1}\Sigma_{T(-m),n}\circ\Tw_{-m}(g)$ as in the proof of Proposition~\ref{prop:PRuniversal}. Then, they form a $R$-system for $T(-m)$. This gives us the relation 
\[
p^{(b-r)N}\sum_{i=1}^Na_ip^{i(r-b)}\cor_{n+i/n}c_{n+i}=0.
\]
Therefore,
\[
\sum_{i=1}^Na_i\cor_{n+i/n}\left(p^{(n+i)(r-b)}c_{n+i}\right)=0
\]
as required.
\end{proof}

\bibliographystyle{amsalpha}
\bibliography{../references}

\end{document}